\documentclass[11pt]{article}
\usepackage{amsmath,amssymb}

\usepackage{times}
\usepackage{bm}
\usepackage{natbib}
\usepackage{algorithm}
\usepackage{amssymb}
\usepackage{mathrsfs}
\usepackage{graphicx}
\usepackage[flushleft]{threeparttable}
\usepackage{enumitem} %%%enumitem

\usepackage{star}

\usepackage[colorlinks,
linkcolor=blue,
anchorcolor=blue,
citecolor=blue
]{hyperref}

\def\##1\#{\begin{align}#1\end{align}}
\def\$#1\${\begin{align*}#1\end{align*}}

\newcommand{\Rom}[1]{\text{\uppercase\expandafter{\romannumeral #1\relax}}}

\usepackage{txfonts}

\usepackage{geometry}

\begin{document}

\title{ \LARGE Sieve Method and Prime Gaps via Probabilistic Method}     % Option 1

%\title{\LARGE  Adaptive Huber Regression: A New Perspective} % Option 2

\author{Buxin Su\thanks{
Department of Mathematics, David Rittenhouse Lab, University of Pennsylvania.209 South 33rd Street, Philadelphia, PA 19104-6395 E-mail: \texttt{subuxin@sas.upenn.edu}. The bulk of this work was performed while BS was a student at University of Toronto.}}

\date{ }

\maketitle

\vspace{-0.25in}

% typeset the title of the contribution
\begin{abstract}
Most prime gaps results have been proven using tools from analytic or algebraic number theory in the last few centuries. In this paper, we would like to present some probabilistic way of proving many essential results. A major component of the proof is a probabilistic approach to the sieve method. In addition, we discuss their connections with Zhang and Maynard's recent work on small and large gaps in prime numbers.
\end{abstract}

\section{Introduction}
In last few centuries, most results in prime gaps was proven via tools from analytic or algebraic number theory. Recently, \cite{zhang2014bounded} proves the large gaps between ${n+1}^{th}$ and $n^{th}$ primes is bounded. That is, 
\begin{equation}
    \liminf_{n} p_{n+1} - p_n \leq 7 \times 10^7
\end{equation}
via a refinement of the work of Goldston, Pintzand, Yıldırım on the small gaps between consecutive primes. The proof strongly relies on a generalization of Bombieri-Vinogradov theorem (See \cite{fouvry1980theorem}). \cite{maynard2015small} have provide alternative approach showing that 
\begin{equation}
    \liminf_{n} p_{n+1} - p_n \leq 600
\end{equation}
Polymath groups also provided several similar results in \cite{polymath2012new}, \cite{castryck2014new} and \cite{polymath2014variants}. However, there are still many unsolved problems in number theory. Goldbach's conjuncture is one the most famous problems. So far, the most recent progress is given by \cite{Chen1973ONTR} half an century ago. 

In this article, mainly focus on Zhang's and Maynard's results, I would like to discuss about small and large gaps between primes and how Probability Theory play an important role in those researches. By showing the connection between Probability Theory and Analytic Number Theory, I would like to demonstrate that probability method is an effective candidate in the future study. In the last part of the manuscript, I would like to discusses a few specific problems that related to our main topic or probability method.

\section{Preliminary Results}
We start with some well-known theorems and preparation. In the remaining of this paper, we may denote $p_n$ by the $n^{th}$ prime number. The first theorem is one of the most fundamental results in Number theory and it can be proven by contradiction.  
\begin{theorem}[Euclidean Theorem]
There are infinitely many primes.
\end{theorem}

Now I would like to provide a slightly different version of Prime Number Theorem. The proof is based on von Mangoldt function. 
\begin{theorem}[Prime Number Theorem]
$|n : p_n \leq x| = (1+o(1))\frac{x}{logx}$
\end{theorem}

Now, we would like to represent the following elegant theorem regarding the least prime gaps $p_{n+1} - p_n$. It is one of the main theorem in \cite{goldston2009primes}. 
\begin{theorem}[\cite{goldston2009primes}]
$$\liminf_{n \to \infty} {\frac{p_{n+1}-p_n}{log\ p_n}} = 0$$
\end{theorem}

\begin{definition}
For $h_1 < h_2 < \mathellipsis < h_k$, we call $H = \{h_1, h_2, \mathellipsis h_k\}$ is admissible iff for any prime $p$, exist $n$ such that $\prod_{i=1}^k {(n+h_i)}$ is coprime to $p$. Equivalently, $h_1 < h_2 < \mathellipsis <h_k$ which avoid at least 1 congruence class $mod \ p$ for any prime p. 
\end{definition}

In order to provide some intuition about admissible set, we have an interesting theorem dur to J, Engelsma. For further discussion on these results, see \cite{granville2015primes}.

\begin{theorem}[J. Engelsma]
There is an admissible set of size 105 contained in $[0,600]$.
\end{theorem}

The following two definitions play a crucial role in Analytic Number Theory. 
\begin{definition}
Define the primorial $\mathit{P(n)}$ of n by the product of all primes $\leq n$, i.e. $\mathit{P(n)}$ = $\prod_{p \leq n} {p}$.
\end{definition}

\begin{definition}
For $(a,q)=1$,$$\pi(x;p,a)=card\{p \leq x: p = a (mod\ q\}$$
define $E_q$ to be $$E_q = \sup_{(a,q) =1} |\pi(x;p,a)- \frac{\pi(x)}{\phi(q)}|$$
we say the primes have ‘level of distribution $\theta$' if for any $A > 0$, $$\sum_{q<x^\theta} {E_q} \ll_A \frac{X}{(logX)^A}$$

Equivalently, For $0<\theta<1$, denote $EH[\theta]$:$$\sum_{q \leq x^\theta} {\sup_{a \in (Z/qZ)^*} {|\sum_{\substack{n\leq X\\ n=a(mod\ q)}} {\Lambda(n)}} - \frac{1}{\phi(q)}\sum_{n\leq X} {\Lambda(n)}}| \ll X(logX)^{-A}$$ for all $A>0$. 
\end{definition}

The following results is a major result of analytic number theory, obtained in the mid-1960s, concerning the distribution of primes in arithmetic progressions, averaged over a range of moduli. The proof details can be found in \cite{tenenbaum2015introduction} and \cite{bombieri1987big}. 

\begin{theorem}[Bombieri-Vinogradov Theorem]
The primes have level of distribution $\theta < \frac{1}{2}$.
\end{theorem}

Finally, we consider Dickson-Hardy-Littlewood Conjectures (see \cite{zhang2009notes}).

\begin{definition}[Dickson-Hardy-Littlewood]
We define DHL(k,j) by the following: given $1\leq j \leq k$, for all admissible k-tuple $h_1<h_2<\mathellipsis<h_k$, exist $\infty$ many $n$ such that $n+h_1<n+h_2<\mathellipsis<n+h_k$ contain more than $j$ primes. 
\end{definition}

\section{Small Gaps Between Primes}
Let X be a large number, in this section, I would like to consider the least prime gap $p_{n+1}-p_n$ in $[X,2X]$. First, we would like to present a new probabilistic proof of the following proposition by Prime Number Theorem and pigeon hole principle. 
\begin{proposition} \label{3.1}
For X large enough, exist prime gaps in $[X,2X]$ of length $\ll logX$.
\end{proposition}

\begin{proof}[Probabilistic Proof of Proposition \ref{3.1}]
Let constant $H \sim logX$ for X given in the proposition. Pick n uniformly random in [X,2X]. Then I have $$\mathbb{P}(n+i\ prime) \sim \frac{1}{logX}$$  for $i = \{0,1,2, \mathellipsis H\}$ by Prime Number Theorem. Add them up, I obtain $$\sum_{h \leq H} {\mathbb{P}(n + h\ prime)} > 1.$$ By Probability pigeon hole principle, I have prime gaps $\leq log\ x$, which is the same result as we obtained using Prime Number Theorem. 
\end{proof}

\begin{remark}
Although we get the same result with even longer proof, the good aspect of this method is it can be further generalized. 
\end{remark}
\begin{remark}
One way of generalize such argument is simply by the observation that there is nothing special about uniformly random among 0 to $H$. Indeed, lots of randomly picking can be applied to this argument. So, an important task would be find the optimal one. 
\end{remark}

Now, I would like to provide several results that highly related to probability theory. In particular, the first two theorems can be proven only involving probability and basic analysis. The proof detail can be found in \cite{panchenko2018introduction}.

\begin{theorem}[Hardy-Ramanujan Theorem]
let $\omega(n)$ denote the number of distinct prime factors of $n$, i.e. $$\omega(n) = card \{p\leq n: p\ is\ prime, p|n \}$$
For any sequence $a_n$ such that $a_n \to \infty$ as $n \to \infty$, the proportion of numbers $N \in \{1, 2, \mathellipsis n\}$ that satisfy $$|\omega(N) - \log\log\ n| \leq a_n*\sqrt{\log\log\ n}$$ goes to 1.
\end{theorem}

The direct consequence of Hardy-Ramanujan Theorem would be Merten's first and second theorem.

\begin{theorem}[Merten's Theorem]
As $n\to \infty$, the following hold:
$$\sum_{p\leq n} {\frac{\log p}{p}}= \log(n) + O(1)$$
$$\sum_{p\leq n} {\frac{1}{p}} = \log(\log(n)) + O(1)$$
where p are primes $\leq n$.
\end{theorem}

The following theorem is another generalization of Hardy-Ramanujan Theorem, which is known as Erdős–Kac theorem, or fundamental theorem of probabilistic number theory. It was first proven in \cite{erdos1940gaussian}. As we may notice, the format is very similar to Central Limit Theorem. 

\begin{theorem}[Erdős–Kac theorem]
For $\omega(n)$ we define above and fixed $a<b$, I have following $$\lim_{x \to \infty} {\frac{1}{x}*card\{n \leq x: a\leq \frac{\omega(n) - \log\log\ n}{\sqrt{\log\log\ n}} \leq b\}} = \Phi(a,b)$$
where $\Phi(a,b)$ is the normal distribution.
\end{theorem}

Our goal now is to show the following theorem from \cite{zhang2014bounded}. The results has been further generalized in \cite{polymath2014bounded}. Zhang's result actually can be seen as a generalization of the method we discuss above.
\begin{theorem}
Let $H = \{h_1, h_2, \mathellipsis h_k\}$ is admissible and $k$ sufficient large, then exist infinite many $n$ such that at least 2 of $\{n+h_1, n+h_2, \mathellipsis n+h_k\}$ are primes. 
\end{theorem}

\begin{remark}
What we really would like to show is for large enough X, exist $n \in [X,2X]$ such that at least 2 of $\{n+h_1, n+h_2, \mathellipsis n+h_k\}$ are primes.
\end{remark}

\begin{remark}
There are two ways to think about this problem. The first one is to generalize the argument I used in the proposition and conclude using pigeon hole principle in Probability. Alternately, we observe that it is sufficient to show for some $n \in [X,2X]$, I have$$\sum_{i=1}^k {1_{n+h_i \ prime}} \geq 2$$ which is equivalent to $$\sum_{i=1}^k {1_{n+h_i \ prime}} - 1 > 0$$ It turn out two ways of thinking have the same consequence as following. 
\end{remark}

our new goal is to find optimal $f(n) > 0$ for $n \in [X,2X]$, where $f(n)$ can be viewed as probability distribution or weight. Also, in order to conclude the claim, we also require $f(n)$ to satisfy $$\sum_{X<n<2X} {f(n)*\Bigg(\sum_{i=1}^k {1_{n+h_i \ prime}}-1 \Bigg)} > 0$$ Once we have this, then at least one of n such that $$\sum_{i=1}^k {1_{n+h_i \ prime}}-1 > 0$$ since the right hand side take value in integer, this gives us exactly what we want. 

Fortunately, we already had a smart choice of $f(n)$ by \cite{goldston2009primes}, we pick $f(n)$ as following: $$f(n)=\Bigg(\sum_{\substack{d\leq x^b\\ d|(n+h_1)(n+h_2) \mathellipsis (n+h_k)}} {\frac{1}{(k+l)!}*\mu(d)*(log\frac{x^b}{d})^{k+l}} \bigg)^2$$ where $l>0$ depends on $k$ and $b$ is a constant in $(0,\frac{1}{2})$

\begin{remark}
Here, we denote the element inside the sum by $\lambda_d$. In fact such choice of weight is found by Goldston, D. A., Pintz, J. and Yildirim (GPY). Basically, the motivation of this choice is $\sum_{d|n} {\mu(d)*(log\frac{n}{d})^k}$ would vanish when $f(n) > k$.
\end{remark}

\begin{remark}
In Zhang's paper, instead of using indicator function, he tends to use $\theta$, which defined to be $\theta(n)=log(n)$ for n is prime and 0 otherwise.However, in the end this makes no difference. 
\end{remark}

Now, we would like check our definition of $f(n)$ does satisfy the requirement above. We obtain the following:

\begin{equation}
    \begin{aligned}
        &\sum_{X<n<2X} {f(n)* \Bigg(\sum_{i=1}^k {1_{n+h_i \ prime}}-1\Bigg)}\\ &=\sum_{X<n<2X} {f(n)* \Bigg(\sum_{i=1}^k {1_{n+h_i \ prime}} \Bigg)} - \sum_{X<n<2X} {f(n)}
    \end{aligned}
\end{equation}

Now, our task become computing two part respectively. It turn out that if square our all the term and rearrange, we have $$\sum_{X<n<2X} {f(n)} = \sum_{d_1,d_2<x^b}{\lambda_{d_1} \lambda_{d_2} *\Bigg(\sum_{\substack{X \leq n<2X\\ gcd(d_1,d_2)|(n+h_1)(n+h_2) \mathellipsis (n+h_k)}} {1}\Bigg)}$$
Adding the characteristic function, I have
$$\sum_{X<n<2X} {f(n)1_{n+h_i \ prime}} = \sum_{{d_1,d_2}<x^b}{\lambda_{d_1} \lambda_{d_2} *\Bigg(\sum_{\substack{X<n<2X\\ gcd(d_1,d_2)|(n+h_1)(n+h_2) \mathellipsis (n+h_k)}} {1_{n+h_i \ prime}}\Bigg)}$$

\begin{remark}
In fact, before we explicitly compute the result above, roughly speaking, second term in the right hand side count how many primes are there in arithmetic progression. So, we have to obtain some results about number of primes in arithmetic progression to control this. Dirichlet theorem is clearly one of the results, but it is too weak to estimate the right hand side. 
\end{remark}

Subtract these two expression above and compute, in Zhang's notation, we define $R(x;d,c)$ to be $$R(x;d,c)=\sum_{\substack{X \leq n<2X\\ n=c(mod\ d)}} {\Lambda(n)} - \frac{X}{\phi(d)}$$ where $\Lambda(n)$ is the von Mangoldt function. Then, E is define to be $$E = \sum_{d<x^{2b}}{|\mu(d)|\sum_{c \in C_i(d)} {|R(x;d,c)|}}$$ where the $C_i(d)$ is defined as $$C_i(d) = \Bigg\{c:1 \leq c \leq d, (c,d) =1, \prod_{j=1}^k {c-h_i+h_j}=0(mod\ d) \Bigg\}$$ Then, the final computation would gives us 
$$\sum_{X \leq n<2X} {f(n)*\Bigg(\sum_{i=1}^k {1_{n+h_i \ prime}}-1\Bigg)} = T*X+O(E)$$ for some expression T and E defined above.

Now, our main tasks become the control of E and T. In order to make the left hand side positive, we only need to have (i)$T > 0 $; (ii) $|E|$ is small enough. By small enough, I would like to show $$E \ll X(logX)^{-A}$$ for large constant A.

Regarding two requirement above, we have some results from \cite{goldston2009primes} and Bombieri-Vinogradov theorem:

\begin{enumerate}
    \item if b $>$ $\frac{1}{4}$, say $c = \frac{1}{4} + \epsilon$, then $T>0$, (i) hold.
    \item if $b < \frac{1}{4}$, $|E|$ is relatively small, (ii) hole. 
\end{enumerate}

\begin{remark}
Eventually, we want to obtain a better choice of b in our second point, i.e. we want to find some $b > \frac{1}{4}$ such that $|E|$ is still relatively small. However, even if we assume Generalized Riemann Hypothesis (GRH), the dilemma above still does not being solved. GRH would tell us that for any R, this error E is smaller up to the form $\sqrt{X}$. If $b>\frac{1}{4}$, than we have at least $\sqrt{x}$ terms in the sum of E with each of them have upper bound $\sqrt{X}$. This is still not efficient bound of E.
\end{remark}

Zhang's main contribution is focused on showing that we are able to efficiently bound $E$. It turn our that $E$ is relatively small, i.e. (ii) hold, for $b=\frac{1}{4}+\frac{1}{1168}$. Traditionally, the estimation of term $R$ in $E$ involve zero of Dirichlet L-function. However, in Zhang's paper, Dirichlet L-function only appear potentially for a few times.

In order to estimate the error term, we first impose some constrain to $d$ that $d$ is divide $P$ where $P$ is the product of primes less than a small power of $X$. This is equivalent to saying that every prime factors of $d$ is less than a small power of $X$. As the result of this constrain, $d$ is not small but every factor of d is small. Having such constrain, we reduce to estimate the sum of $|\Delta(\gamma,d,c)|$ where $\gamma$ is supported on $[X,2X)$ and $\Delta$ is defined to be $$\Delta(\gamma,d,c) = \sum_{n = c (mod\ d)} {\gamma(n)} - \frac{1}{\phi(d)}*\sum_{(n,d)=1} {\gamma(n)}$$.

If $d|P$ and $d$ is not too small, say $d > x^{\frac{1}{2}-\epsilon}$, then d has factorization $d=qr$ where the range of r can be very flexibly chosen. Applying combinatorial argument, we reduce the estimation of $|\Delta(\gamma,d,c)|$ to three types of Dirichlet convolution.

\begin{enumerate}
    \item first two type: $\gamma = \alpha * \beta$ the Dirichlet convolution of two function, where $\beta$ is supported on $[N,2N)$. When $N>x^{\frac{3}{8}}$, we can reduce to estimate the sum of $$\sum_{Q<q<2Q} {\sum_{R<r<2R} {|\mu(qr)|*\sum_{c \in C_i (qr)} {\Delta(\gamma,qr,c)}}}$$ If we choose $R$ to be slightly smaller than $N$, then we are able to conclude using Weil’s bound for Kloosterman sums.
    \item third type: $\gamma = \alpha * \beta_1 *\beta_2 *\beta_3$ Note in Zhang's paper, instead of using $\beta_i$, he uses $\mathcal{X}_i$ to emphasis that these functions are characteristic function of certain intervals. Applying $d=qr$ to this type, we can get an efficient upper bound using Deligne’s proof of the Riemann Hypothesis for varieties over finite fields.
\end{enumerate}

After the effort above, we are able to get an efficient bound and this would end Zhang's argument. Different from Zhang's method, the method about choosing $f(n)$ is further generalized by J. Maynard in \cite{maynard2015small}.

Instead of using $f(n)$ as above, Maynard generalize to "multi variable" case, which allows to conclude stronger statement. 

\begin{theorem}[\cite{maynard2015small}]
The small gaps between primes are given by
$$\liminf_{n} {p_{n+m}-p_n} \leq Cm^3*e^{4m+5}$$
for all $m \in \mathbb{N}$ and morover,
$$\liminf_{n} {p_{n+1}-p_n} \leq 600$$
\end{theorem}

\begin{remark}
In Zhang's result, we are able to  see that the main breakthrough is about the primes in Arithmetic Progression. However, in Maynard results, what we would do is to modify GPY sieve. It turn out that this would cause the optimization problem and some combinatorial problem. 
\end{remark}

First, we consider the same function as we did above, let $$S=\frac{\sum_{N\leq n<2N} {1_{n+h_i prime}*f(n)}}{\sum_{N\leq n<2N} {f(n)}}$$

\begin{remark}
As we discuss above, we need $S>1$ for bound gaps. Also, if $S>m$ for all large $N$, then $\liminf {p_{n+m} - p_n} < \infty$
\end{remark}

As we said above, we used to define $$f(n)=(\sum_{\substack{d\leq R\\ d|(n+h_1)(n+h_2) \mathellipsis (n+h_k)}} {\lambda_d})^2$$ where $\lambda_d = \mu(d)*g(d)$ for $g(d) \approx (log\frac{R}{d})^{k+l}$

Maynard provides us a new choice of $f(n)$ as following 
$$f(n)=\Bigg(\sum_{\substack{d_1, d_2, \mathellipsis d_k\\ d_1*\mathellipsis*d_k < R\\ d_i|(n+h_1)(n+h_2) \mathellipsis (n+h_k)}} {\lambda_{d_1,d_2\mathellipsis,d_k}}\Bigg)^2$$
where $\lambda_{d_1,d_2\mathellipsis,d_k} = \mu(\prod_{i=1}^k d_i)*h(d_1,d_2, \mathellipsis, d_k)$. The function $h$ is defined in term of a smooth function F that we are free to choose. One advantage of such choice of weight is that our weight depends on divisor of each of $n+h_i$.\\

Recall the theorem above has two part, they deal with $p_{n+m}$ and $p_{n+1}$ respectively. It turn out function $F$ can be slightly different for two part. Thus, before we find the function $F$, we provide some definition and theorem. Let $F:[0,1]^k \to \mathit{R}$ denote a nonzero square-integrable function with support in $R_k = \{(x_1,x_2, \mathellipsis, x_k) \in [0,1]^k : \sum_{i} {x_i} \leq 1. \}$ we define $$I(F) = \int_{[0,1]^k} {F(t_1, t_2, \mathellipsis, t_k)^2 dt_1dt_2\mathellipsis dt_k}$$ $$J(F) = \sum_{i=1}^k {\int_{[0,1]^{k-1}} {(\int_{[0,1]}{F(t_1, t_2, \mathellipsis, t_k)dt_i)^2}}dt_1\mathellipsis dt_{i-1}dt_{i+1} \mathellipsis dt_k}$$  $$M_k = \sup_F \frac{J(F)}{I(F)}$$ 

\begin{theorem}[Main Theorem in \cite{maynard2015small}]
\label{3.17}
For any $0<\theta <1$, if $EH[\theta]$ and $M_k > \frac{2m}{\theta}$, then $DHL[k,m+1]$. Equivalently, let the primes have level of distribution $\theta$. If $M_k > \frac{2m}{\theta}$, then there are infinitely many integer $n$ such that at least $m+1$ of the $n+h_i$ are primes. 
\end{theorem}

\begin{remark}
As we would see in the following passage, the final result about small gaps is based on Bombieri-Vinogradov Theorem that $\theta < \frac{1}{2}$ unconditionally.However, from the theorem above, we are able to see if we assume Elliott-Halberstam conjecture ($\theta < 1$), then we are able to see a better result i.e. $\liminf_{n} {p_{n+1} - p_n} \leq 12$
\end{remark}

So, having theorem above, our problem is to prove $M_k > \frac{2m}{\theta}$, which is question about the lower bound of $M_k$.

\begin{proof}[Proof Sketches of Theorem \ref{3.17}]
Let $$F(t_1,t_2,\mathellipsis,t_k) = \prod_{i=1}^k g(kt_i)$$ if $(t_1,t_2,\mathellipsis,t_k) \in R_k$ and $F(t_1,t_2,\mathellipsis,t_k)=0$ otherwise for some function g. 

If the center of mass of $g^2$ satisfy $$\mu = \frac{\int_{0}^{\infty} {tg(t)^2}dt}{\int_{0}^{\infty} {g(t)^2}dt} <1$$ then by the concentration of measure we expect the restriction on support of $F$ to be negligible. 

If g is supported on $[0,T]$, we find that $$M_k \geq \frac{\int_{0}^{T} {tg(t)^2}dt}{\int_{0}^{T} {g(t)^2}dt} * \Bigg(1-\frac{T}{k(1-\frac{T}{k} - \mu)^2}\Bigg)$$

For fixed $\mu$ and $T$, we optimize over all such $g$ by Calculus of Variation. The optimal $g$ would be $$g(t) = \frac{1}{1+At}$$ if $t \in [0,T]$

With the g above, we find that a suitable choice of $A$, $T$ gives $$M_k > \log\ k -2\log\log\ k -2$$ if $k$ is large enough.

By Bombieri-Vinogradov Theorem, we can take any $\theta < \frac{1}{2}$ unconditional. With the theorem above and some effort, we are able to obtain $$\liminf_{n} {p_{n+m}-p_{n}} \leq Cm^3e^{4m}$$ This complete the first part of theorem. When k is relatively small, we are bale to use a different argument to bound $M_k$ involving symmetric polynomials.\\

Let symmetric polynomial $P(t_1,t_2,\mathellipsis, t_k)$ define as following $$P(t_1,t_2,\mathellipsis, t_k) = \sum_{a+2b \leq d} {c_{a,b}(1-P_1)^a{P_2}^b}$$ where $P_1 = \sum_{i}{t_i}$ and $P_2 = \sum_{i} {t_i^2}$ i.e. symmetric polynomial of degree at most $d$. Now let $$F = P(t_1,t_2,\mathellipsis, t_k)$$ if $(t_1,t_2,\mathellipsis, t_k) \in R_k$ and $F=0$ otherwise. This form of $F$ is relatively simple, so we are able to compute the integral we define above, $I(F)$ and $J(F)$. It turn out that $\frac{J(F)}{I(F)}$ become a ratio of quadratic forms. i.e. we are able to obtain $$M_k \geq \sup_{a \in \mathit{R}^d} {\frac{a^{T}A_2a}{a^{T}A_1a}}$$ for positive definite symmetric rational matrices $A_1, A_2$. From optimization perspective, we are able to obtain that right hand side obtain its maximal value when $a$ equal the largest eigenvalue of ${A_1}^{-1}A_2$.Thus, range over all symmetric polynomial of degree at most $d$, we are able to obtain an efficient bound. For example, if $k=105$; we have $M_k > 4$ and if $k = 5$, we have $M_k > 2$.Note as a special case of the theorem we state before, if $M_k > \frac{2}{\theta}$, then there are infinitely many integer $n$ such that at least $2$ of the $n+h_i$ are primes. Combined with Engelsma and Bombieri-Vinogradov Theorem, we are able to conclude $\liminf_{n} {p_{n+1} - p_n} \leq 600$.
\end{proof}

\section{Discussion}
In this section, we would like to present some though for possible future works. 
\begin{enumerate}
\item
Notice in the beginning of discussion of small gaps between primes, we consider the indicator function and pick $f(n)$ to satisfy $$\sum_{X<n<2X} {f(n)*\Bigg(\sum_{i=1}^k {1_{n+h_i \ prime}}-1\Bigg)} > 0$$ which basically is a weighted sum. The things we really want to show is $\sum_{i=1}^k {1_{n+h_i \ prime}}-1 > 0 $ for some n. The most natural way of thinking would be improve our choice of $f(n)$ or our technical of estimating sum. If fact, both of them can be quite difficult like \cite{zhang2014bounded} and \cite{maynard2015small}. However, in order to obtain $\sum_{i=1}^k {1_{n+h_i \ prime}}-1 > 0 $, it is not necessary for us to use sum or weighted sum. Instead, we can use following\\

consider $$\prod_{X<n<2X} {\Bigg(\sum_{i=1}^k {1_{n+h_i \ prime}}-\frac{3}{2}\Bigg)}$$ Since $\sum_{i=1}^k {1_{n+h_i \ prime}}$ are integer, subtract $\frac{1}{2}$ does not change the sign. Also, for the same reason, each term in the product would not be zero. So, if we can prove the product are positive, there are two cases to consider: 
\begin{enumerate}
    \item if one of ${\sum_{i=1}^k {1_{n+h_i \ prime}}-\frac{3}{2}} > 0$, then we are done for the same reason we said before.
    \item if all of them are negative and due to our $X$, there are even number of negative term. In this case, we can product over either $n \in [X, 2X)$ or $n \in (X, 2X)$. there must be one them that only contain odd number of choice of n. Thus, we do not need to consider this cases and suffice to show the positivity of product. 
\end{enumerate}

More generally, as we did in the sum, we can add some 'weight' $f(n)>0$ to the product to make the calculation easier. $$\prod_{X<n<2X} {f(n)*\Bigg(\sum_{i=1}^k {1_{n+h_i \ prime}}-\frac{3}{2}\Bigg)}$$

Moreover, there are nothing special about product or sum. Let F be a function from $\mathbb{R}^X$ to $\mathbb{R}$ satisfying: if $F(x_1, x_2, \mathellipsis x_X) > 0$, then $x_i > 0$ for some $i \leq X$. For such F consider $$F\Bigg(\sum_{i=1}^k {1_{X+1+h_i \ prime}}-1, \mathellipsis, \sum_{i=1}^k {1_{2X+h_i \ prime}}-1\Bigg)$$ if we can show the expression is positive for some $F$ satisfying the condition we require above, we are done. In this case, we have more flexibility in choice of F since we input only one constrain to our function $F$. As a special cases, weighted sum clearly satisfy the condition we impose on $F$. Thus, weighted sum would be a special cases of $F$. 

\item
In small gaps between primes, Zhang and Maynard focus on different aspects and did the improvement independently. Zhang improves the result in arithmetic progression, whereas Maynard modify GPY sieve. Zhang's result allow us to have a better bound for the estimation and Maynard's method allows us to have a more subtle and stronger results. Thus, is it possible that combine the two results together with some necessary modification to have a better conclusion? For example, using Maynard's sieve, we somehow modify Zhang's results in Arithmetic Progression to have a better bound. 

\item
In the studying of primes gaps, the question we are dealing with is about $\limsup$ or $\liminf$, which basically are question asking infinite often. However, this kind of question has been studied for a long time in Analysis or Probability. Having this in mind, is it possible for us to model the distribution or primes gaps for all primes, so that the results from probability might give us some new way of thinking. Moreover, if we convert problem in prime gaps to probability, we are able to use theorem like Borel-Cantalli lemma, Strong/Weak Law of Large Number or Central Limit theorem. In fact, in the proof of Erdos-Kac theorem we state in the beginning, Central Limit theorem play an important role. The classic version of those probability theorem above require the random variable to be iid. However, there indeed some stronger version that only require independent, which are much easier to obtain by only considering a sub-sequence if necessary. Since we are studying infinite often question, prove the statement on subsequence would gives us the result.\\

For example, we could set random variable to be the number of primes smaller that $n$ or number of primes divide $n$. If we want to study $p_{n+m}-p_n$, we consider the number of k in admissible set. We might set event $$A_k = \{n+\{h_1, h_2, \mathellipsis h_k\}\ contain\ m\ primes\ for\ infinite\ many\ n\}$$ From construction of $A_k$ or Kolmogorov 0-1 law, we are able to see $\mathbb{P}(A_k)$ = 0 or 1. So we can sum up all $k \leq N$, $$\sum_{k\leq N} {\mathbb{P}(A_k)}$$ once we have a result of positive number, the $N$ would be the number we are looking for.

\item
In probability, there are a noted inequality called Azuma's inequality (See \cite{vershynin2018high} for example). It states that for random variable $X_1, X_2, \mathellipsis X_n$ and function $f$, suppose $\mathbb{E}(f(X_1,\mathellipsis X_n)$ exist and $$|f(X_1,\mathellipsis X_n) - f(X_1,\mathellipsis, t_i, \mathellipsis X_n)| \leq a_i$$ for any $t_i$ take value in the range of random variable. Then we have $$\mathbb{P}(f(X_1,\mathellipsis X_n) - \mathbb{E}(f(X_1,\mathellipsis X_n) \geq t) \leq e^{\frac{-t^2}{2\sum_{i}{a_i^2}}}$$

One of the application of Azuma's inequality is Erdo-Renyi graph. In Erdo-Renyi graph, one considers graph $G(n,q)$ with n vertices and each pair of vertices are represent by iid. Bernoulli(p), $e_{ij}$, where $e_{ij} = 1$ if edge $ij$ are connected in the graph and $e_{ij}=0$ otherwise. we send a coloring to the graph $G(n,p)$ and the graph is called admissible if 'color of i $\neq$ color of j if $e_{ij} = 1$'. The question people interested in the chromatic number of $G(n,p)$ denoted as $\chi(G(n,p))$. One important fact is that the expectation of $\chi(G(n,p))$ is in the similar form as Prime Number theorem. i.e. $$\mathbb{E}(\chi(G(n,p))) \sim \frac{n}{\log n}$$ In probability, people are trying to improve the expectation above by substitute the $\log n$ with expression like $\log\log n$ and etc. However, all these have lots of similarity with primes gaps and in fact, we have already lots of results looking like $\log\log n$. So, if we are able to find some connection between chromatic number of Erdo-Renyi graph and primes gaps, we are able to witness a new way of thinking. Even though, at this time, Erdo-Renyi graph might not provide some stronger observation in prime gaps, like what probability sometimes did to Analysis, the intuition and simpler cases are very helpful and enlightening. 

\end{enumerate}

\section{Symmetry Between Small and Large Gaps}
In this section, we mainly discuss Maynard's remarkable work which allow us to talk about large prime gaps using progress on small prime gaps. This method is quite different from the previous ones.

In $[1,X]$ for large X, we use $G(X)$ to denote the largest prime gaps $p_{n+1}-p_n$ in $[1,X]$. We would like to deal with the problem that how does $G(X)$ grow with respect to X. Finding large prime gaps are basically the same thing as finding large consecutive string of composite numbers. Having this in mind, I have following proposition:

Similar to small gaps between primes, by Prime Number Theorem and pigeon hole principle, I have
\begin{theorem}
\label{4.1}
$p_{n+1}-p_n \gg log(p_n) \ i.o.$
\end{theorem}
\begin{remark}
As we did in previous section, there is also an alternate proof of the proposition above using primorial and eventually, it can be further generalized. 
\end{remark}

\begin{proof}[Probabilistic proof of Theorem \label{4.1}]
consider the primorial $P(n)$, we obtain $$P(n)+2, P(n)+3, \mathellipsis P(n)+n$$ are consecutive string of composite numbers. Thus I conclude $p_{n+1}-p_{n} \gg log\ p_n$
\end{proof}

For interval $\{2,3,\mathellipsis,n\}$ we are able to cover all element by residue class of all primes less than $n$. More specially, first we delete all number that $0 (mod\ 2)$, then delete all number that $0 (mod\ 3)$ until all number that $0 (mod\ p)$. Generalizing the method above, for interval $\{a, a+1, \mathellipsis, b\}$, suppose the interval can be covered by primes of residue class $$c_2 (mod\ 2), c_3 (mod\ 3), \mathellipsis c_p (mod\ p)$$
for $p\leq n$. Using Chinese Remainder theorem, we can find a $y$ such that 
$$y = -c_2 (mod\ 2), y=-c_3 (mod\ 3), \mathellipsis, y=-c_p (mod\ p).$$ 
Then $y + \{a, a+1, \mathellipsis, b\}$ would a string of composite numbers. In particular, every number is divisible by primes that are less than $n$.

Base on the same method as above, use prime $p \leq n $, we can cover $$\{1,2,3,\mathellipsis, y\}$$ where $y = n*(\frac{\log n  \log\log\log n}{\log n})$. In order to show this, we have find some smart way to pick our $c_i$ residue class.

The essential 'truncation' in the choice of $c_i$ is the number $\frac{n}{2}$. Roughly speaking, for $p\leq \frac{n}{2}$, pick some congruence class that cover all element except primes between $\frac{n}{2}$ and $y$. For $p\geq \frac{n}{2}$, we try to cover remaining prime. Thus, the question turn out to be given $\frac{n}{2} \leq p \leq n$, how to find $c_p (mod\ p)$ that would cover at many as primes in $[1,y]$ as possible?

Originally, people focus on making a chain in the form of $\{c_p, c_p + p, c_p + 2p, \mathellipsis\}$ to be all primes. However, in order to solve this question, one do not really need all of them to be primes. While Maynard works on small gaps between primes, he find a way to produce many primes in $\{n+h_1, n+h_2, \mathellipsis, n+h_k\}$. Thus, the question above become a special case of the result in small gaps. This would finally be able to prove the our claim.

\newpage
\bibliographystyle{ims}
\bibliography{references}

\end{document}